\documentclass[10pt,reqno,a4paper,draft]{amsart}
\usepackage{amssymb}
\usepackage{a4wide}
\usepackage{hyperref}

\usepackage{amscd,mathrsfs}
\input xy
\xyoption{all}

\newcommand{\Cc}{\mathcal{C}}
\newcommand{\Co}{\mathcal{C}_0}
\newcommand{\Cinf}{\mathcal{C}^{\infty}}
\newcommand{\A}{\mathcal{A}}
\newcommand{\G}{\mathcal{G}}
\newcommand{\Gld}{\mathcal{G}_{\rm{ld}}}
\newcommand{\Xct}{\widetilde{X}_c}
\newcommand{\Yct}{\widetilde{Y}_c}
\newcommand{\K}{\mathbb{K}}
\newcommand{\Kt}{\widetilde{\mathbb{K}}}
\newcommand{\EM}{\mathcal{E}_M}
\newcommand{\Neg}{\mathcal{N}}
\newcommand{\eps}{\varepsilon}
\newcommand{\R}{\mathbb{R}}
\newcommand{\C}{\mathbb{C}}
\newcommand{\N}{\mathbb{N}}

\newcommand{\Der}{\operatorname{Der}}
\newcommand{\pr}{\operatorname{pr}}

\title[An algebraic approach to manifold-valued generalized functions]{An algebraic approach\\to manifold-valued generalized functions}
\date{May 27, 2011} % First Version: Oct 18, 2010

\numberwithin{equation}{section}

\theoremstyle{plain}
\newtheorem{theorem}{Theorem}[section]

\newtheorem{proposition}[theorem]{Proposition}
\newtheorem{lemma}[theorem]{Lemma}
 
\theoremstyle{definition}
\newtheorem{definition}[theorem]{Definition}

\theoremstyle{remark}
\newtheorem{remark}[theorem]{Remark}

\begin{document}

\author{Annegret Burtscher}
\address{Faculty of Mathematics, University of Vienna\\
Nordbergstr.~15, 1090 Vienna, Austria}
\email{\href{mailto:annegret.burtscher@univie.ac.at}{annegret.burtscher@univie.ac.at}}
\thanks{This work was supported by project P20525 of the Austrian Science Fund (FWF) and research stipend FS 506/2010 from the University of Vienna.}


\begin{abstract}
We discuss the nature of structure-preserving maps of varies function algebras. In particular, we identify isomorphisms between special Colom\-beau algebras on manifolds with invertible manifold-valued generalized functions in the case of smooth parametrization. As a consequence, and to underline the consistency and validity of this approach, we see that this generalized version on algebra isomorphisms in turn implies the classical result on algebras of smooth functions.
\end{abstract}
\maketitle
\thispagestyle{empty}

\section{Introduction}
\label{intro}

 In order to study problems in geometric analysis on manifolds, the set of manifold-valued generalized functions with domain manifold $X$ and target manifold $Y$, denoted by $\G[X,Y]$, has been introduced by M.~Kunzinger in \cite{Kunzinger2002c,Kunzinger2002b}. Many useful properties have now been investigated, e.g., intrinsic characterizations are provided in \cite{Kunzinger2003} and it is known that $\G[X,Y]$ is a sheaf and contains an embedded copy of the space of continuous mappings $\Cc(X,Y)$~\cite{Kunzinger2009}.

 In addition, we may view a manifold-valued generalized function $\varphi \in \G[Y,X]$ as an algebra homomorphism $\varPhi$ (defined by the pullback under $\varphi$) between the special Colombeau algebras $\G(X)$ and $\G(Y)$ on $X$ and $Y$, respectively. This algebraic approach is indeed fruitful since algebra homomorphisms $\varPhi$ from $\G(X)$ to $\G(Y)$ in turn uniquely define manifold-valued generalized functions. Bijective manifold-valued generalized functions from $Y$ to $X$, in particular, can be identified with algebra isomorphisms between the corresponding Colombeau algebras on $X$ and $Y$ \cite{Vernaeve2010,Burtscher2009}.

 It is the aim of this paper to elaborate the ideas for the correspondence between algebra morphisms $\G(X) \to \G(Y)$ and manifold-valued generalized functions $\G[Y,X]$ with smooth parameter dependence developed in the author's thesis~\cite{Burtscher2009}, and to point out similarities and differences to the case of generalized functions with arbitrary parametrization~\cite{Vernaeve2010}. Moreover, to come full circle, we will see that such a characterization of isomorphisms on algebras of generalized functions in fact implies the standard result on algebras of smooth functions.

\section{Background}
\label{sec:1}

 Starting with a prototypical example in Banach algebra theory we shall underline the importance of representing topological spaces by algebras. Suppose we are given a locally compact Hausdorff space $X$, then the set of all continuous real- or complex-valued functions on $X$ that vanish at infinity, henceforth denoted by $\Co(X)$, forms a commutative $C^\ast$-algebra under pointwise addition and multiplication. The spectrum $\widehat{\A}$ of a commutative Banach algebra $\A$ is the set of all non-zero multiplicative linear functionals. In our case the spectrum $\widehat{\Co(X)}$ is homeomorphic to $X$ which implies the subsequent result, see e.g.~\cite[Thm.~1.3.14]{Higson2000}.

\begin{theorem} \label{thm:cont}
 Let $X$ and $Y$ be locally compact Hausdorff spaces. The algebras $\Co(X)$ and $\Co(Y)$ are algebraically isomorphic if and only if $X$ and $Y$ are homeomorphic. Such an isomorphism is in fact an isometry.
\end{theorem}

 In particular, for $X$ and $Y$ compact, we deduce that $\Cc(X)$ and $\Cc(Y)$ are algebraically isomorphic if and only if $X$ and $Y$ are homeomorphic.

 A similar result holds true for algebras of smooth functions $\Cinf(X)$. The standard proofs of this make use of `Milnor's exercise'~\cite[p.~11]{Milnor1974} and thereby of the identification of multiplicative functionals in the algebras of smooth functions with points on the manifold, see e.g.~\cite[Suppl.~4.2C]{Abraham1988}. In 2003 A.~Weinstein pointed out that such proofs strongly rely on the the fact that the manifolds are assumed to be second countable, and formulated a theorem that was finally proved independently by J.~Mr\v{c}un and J.~Grabowski~\cite{Mrcun2005,Grabowski2005} by purely algebraic approaches.

\begin{theorem} \label{thm:smooth}
 Let $X$ and $Y$ be any Hausdorff smooth manifolds (not necessarily second countable, paracompact or connected). Then any algebra isomorphism $\Cinf(X) \rightarrow \Cinf(Y)$ is given by composition with a unique diffeomorphism $Y \rightarrow X$.
\end{theorem}

 In order to formulate a similar result for generalized functions, such as distributions, additional constructions are necessary. Due to L.~Schwartz's impossibility result~\cite{Schwartz1954} distributions can not be multiplied in a way that is consistent with the classical pointwise multiplication of continuous functions without dropping desirable algebraic properties. The theory of generalized functions, initiated by J.F.~Colombeau~\cite{Colombeau1984,Colombeau1985}, resolves this problem of non-multiplicativity by embedding the space of distributions in an associative and commutative differential algebra, while preserving the pointwise multiplication of smooth functions. These so-called Colombeau algebras have been applied to many problems, primarily for non-linear partial differential equations, as well as the study of non-smooth differential geometry. Structure preserving maps between paracompact manifolds in Colombeau theory are the so-called compactly bounded (c-bounded) generalized functions~\cite{Kunzinger2002c,Kunzinger2003,Kunzinger2009}. Recently H.~Vernaeve~\cite{Vernaeve2010} established a correspondence analogous to the above theorems between manifold-valued generalized functions and the algebra homomorphisms of Colombeau algebras.

\begin{theorem} \label{thm:vern}
 Let $X$ and $Y$ be second countable Hausdorff manifolds.
\begin{enumerate}
\item An algebra morphism $\G(X) \rightarrow \G(Y)$ is, up to multiplication by idempotents $e \in \G(X)$, uniquely determined by a locally defined c-bounded generalized function $\Gld[Y,X]$.
\item Every algebra isomorphism $\G(X) \rightarrow \G(Y)$ is given by composition with an invertible locally defined c-bounded generalized function $\Gld[Y,X]$.
\end{enumerate}
\end{theorem}

 As for the classical results, the approach is based on algebraic properties of non-zero multiplicative linear functionals $\nu: \G(X) \rightarrow \Kt$, where $\Kt$ denotes the set of generalized numbers ($\mathbb{K}$ either $\mathbb{R}$ or $\mathbb{C}$). The compactly supported generalized points in $X$ are identified with the ideals $\ker(\nu) \lhd \G(X)$. These ideals $\ker(\nu)$ are, however, not maximal since $\Kt$ is not a field~\cite[Prop.~6.1.6]{Burtscher2009}.

 Another difference from the classical situation occurs in the use of locally defined c-bounded generalized functions $\Gld[Y,X]$. Due to the Whitney Embedding Theorem~\cite{Hirsch1976} we can assume, without loss of generality, that $X$ and $Y$ are submanifolds of some $\R^m$ and $\R^n$, respectively. Since generalized functions in $\G(Y)^m$ that are c-bounded into $X$ may not entirely map to $X$, they do not necessarily define an element in $\G[Y,X]$ but only in $\Gld[Y,X]$.

 Such intricacies may be avoided by restricting to Colombeau generalized functions with smooth parametrization~\cite[Rem.~2.4]{Kunzinger2009}. Theorem~\ref{thm:vern} still holds, but $\Gld[Y,X]$ may be replaced by $\G[Y,X]$. Besides, no idempotents in $e$ (other than $0$ and $1$, and combinations of both on different connected components) appear in this setting. Thus Colombeau algebras with smooth parametrization seem more `natural' and consistent with regard to geometric problems, cf.\ also~\cite{Kunzinger2009}.

\section{Preliminaries}
\label{sec:2}

 Throughout this paper, manifolds are assumed to be finite dimensional, smooth, Hausdorff and second countable. Let $I=(0,1]$ and denote by $\Cinf(A,B)$ the smooth functions from $A$ to $B$. If $B=\mathbb{K}$ ($\mathbb{K}$ being either $\R$ or $\mathbb{C}$) then we write $\Cinf(A)$.

 The \emph{special Colombeau algebra} on $X$, $\G(X)$, is defined as the quotient $\EM(X) / \Neg(X)$ of the sets of moderate and negligible functions:
\begin{align*} 
     \EM (X) := \, & \lbrace (u_{\eps})_{\eps} \in \Cinf(I \times X) \; | \; \forall K \subset \subset X \, \forall P \in \mathcal{P}(X) \, \exists N \in \N: \\
             & \sup_{x \in K} |Pu_{\eps}(x)| = O(\eps^{-N}) \text{ as } \eps \rightarrow 0 \rbrace \\
     \Neg (X) := \, & \lbrace (u_{\eps})_{\eps} \in \EM(X) \; | \; \forall K \subset \subset X \, \forall m \in \N: \\ 
             & \sup_{x \in K} |u_{\eps}(x)| = O(\eps^m) \text{ as } \eps \rightarrow 0 \rbrace ,
\end{align*}
 where $\mathcal{P}(X)$ is the set of linear differential operators on $X$. The space of distributions, $\mathcal{D}'(X)$, can be linearly embedded in $\G(X)$, and $\Cinf(X)$ is a faithful subalgebra of $\G(X)$~\cite{Grosser2001}.
 The scalars are called \emph{generalized numbers}. They form a ring, $\Kt$, defined as the quotient $\EM / \Neg$:
\begin{align*} 
   \EM & := \lbrace (r_{\eps})_{\eps} \in \Cinf(I) \; | \; \exists N \in \N: |r_{\eps}| = O(\eps^{-N}) \text{ as } \eps \rightarrow 0 \rbrace \\
   \Neg & := \lbrace (r_{\eps})_{\eps} \in \Cinf(I) \; | \; \forall m \in \N: |r_{\eps}| = O(\eps^{m}) \text{ as } \eps \rightarrow 0 \rbrace .
\end{align*}

 The space $\G[X,Y]$ of \emph{compactly bounded (c-bounded) Colombeau generalized functions} on $X$ with values in the manifold $Y$ is similarly defined by an equivalence relation on the space $\EM[X,Y]$ of moderate c-bounded maps where the c-boundedness condition~\cite{Kunzinger2002c,Kunzinger2003,Kunzinger2009} takes the form \[ \forall K \subset\subset X \, \exists L \subset\subset Y \, \exists \eps_0>0 \text{ such that } \forall \eps<\eps_0: \, u_{\eps}(K) \subseteq L. \]
 In absence of a linear structure on the target space $Y$, the equivalence relation is more involved than in the definition of $\G(X)$.

 The set of \emph{compactly supported points} in $X$ (and $Y$), $\Xct$ (and $\Yct$), allows for point value characterizations of generalized functions in $\G(X)$ (and $\G[X,Y]$)~\cite{Grosser2001,Kunzinger2003,Nigsch2006}. A net $(x_{\eps})_{\eps} \in \Cinf(I,X)$ is called compactly supported if there exists a compact set $K \subseteq X$ such that $x_{\eps} \in K$ for $\eps$ sufficiently small. Two nets $(x_{\eps})_{\eps}, (y_{\eps})_{\eps} \in \Cinf(I,X)$ are called equivalent if, for any Riemannian metric $\mathbf{g}$, the distance is $d_\mathbf{g}(x_{\eps},y_{\eps}) = O(\eps^m)$ for any $m \in \N$ as $\eps$ tends to $0$. We shall denote by $\Xct$ the set of all equivalence classes---with respect to the above equivalence relation---of compactly supported points in $X$.

\section{Isomorphisms of algebras of Colombeau generalized functions}
\label{sec:3}

 A manifold-valued generalized function $\varphi \in \G[Y,X]$ naturally defines an algebra homomorphism $\varPhi: \G(X)\rightarrow \G(Y)$ via composition, i.e.\
$$\varPhi(u) = u \circ \varphi~\text{for all}~u \in \G(X).$$
On the other hand, as mentioned above, H.~Vernaeve~\cite{Vernaeve2010} provided Theorem~\ref{thm:vern} characterizing morphisms between special Colombeau algebras (with non-smooth parametrization) with locally defined c-bounded generalized functions. As in the classical case, the idea was to identify `points' in a manifold $X$ with `algebraic objects' in the algebra $\G(X)$---and similarly for $Y$---in order to construct a `structure preserving map' between $Y$ and $X$ given an algebra homomorphism between $\G(X)$ and $\G(Y)$. By doing so, multiplicative linear functionals $\nu: \G(X) \rightarrow \Kt$ could be associated with compactly bounded generalized points $\Xct$ in $X$ \cite[Thm.~4.5]{Vernaeve2010}.

 In contrast to the classical case for algebras of continuous and smooth functions, however, the scalars in these algebras of generalized functions only form a ring and not a field. An immediate consequence of this is that the ideals $\ker(\nu)$ in $\G(X)$ are not maximal. H.~Vernaeve therefore introduced a new notion of `invertibility' and new `maximal ideals', that---with slight modifications---can also be carried over to Colombeau generalized functions with smooth parametrization as defined in Section~\ref{sec:2}.

\begin{definition} \label{def:Sinv}
 Let $S \subseteq I$ such that $0 \in \overline{S}$ and let $\A$ denote the algebra $\Kt$ or $\G(X)$. An element $u \in \A$ is called \emph{invertible with respect to $S$} if there exists $v \in \A$ and $r \in \Kt$ such that $$uv = r1 \text{ in } \A \text{ and } \left. r \right\vert_S = 1 \text{ in } \Kt.$$
\end{definition}

 The restriction $\left. r \right\vert_S = 1$ is to be understood at the level of representatives as functions of $\eps$, i.e.\ for a representative $(r_{\eps})_{\eps} \in \EM$ of $r$ and some $(n_{\eps})_{\eps} \in \Neg$ we require that $r_{\eps} = 1 + n_{\eps}$ for all $\eps \in S$. Note that this condition in the setting of \cite{Vernaeve2010}, together with $\left. r \right\vert_{S^c} = 0$, characterizes the idempotents in $\Kt$~\cite{Aragona2008}. In our definition of $\Kt$, however, there are no idempotents and this type of invertibility had to be constructed artificially. Still, the same properties hold:

\begin{proposition}
 Let $r \in \Kt$. Then $r \neq 0$ if and only if there exists $S \subseteq I$, $0 \in \overline{S}$, such that $r$ is invertible with respect to $S$.
\end{proposition}
\begin{proof}
 We first observe that an element $r \in \Kt$ is invertible with respect to $S$ if and only if it is strictly non-zero on $S$. This is an analogue to the standard case $S = I$, cf.\ [9, Thm.~1.2.38]. For $s$ an $S$-inverse of $r$ it immediately follows on the level of representatives from
$$\left\vert r_{\eps} \right\vert = \left\vert \frac{1 + n_{\eps}}{s_{\eps}} \right\vert \geq \frac{1}{2 \left\vert s_{\eps} \right\vert} > \eps^N$$
 for an $N \in \N$ and $\eps \in S$ sufficiently small. Let, on the contrary, $r$ be strictly non-zero on $S$, i.e.\ there exists a representative $(r_{\eps})_{\eps}$ of $r$ and $m \in \N$ such that $\left\vert r_{\eps} \right\vert > \eps^m$ for $\eps \in S$ sufficiently small. Denote by $\chi: I \rightarrow \R$ the bump function that equals $1$ on $\{ \eps : \left\vert r_{\eps} \right\vert \geq \eps^m \}$ and $0$ on $\{ \eps : \left\vert r_{\eps} \right\vert \leq \eps^{m+1} \}$. Set $s_{\eps} = 0$ on $\{ \eps : r_{\eps} = 0\}$ and $s_{\eps} = \frac{\chi(\eps)}{r_{\eps}}$ else. Then $(s_{\eps})_{\eps} \in \EM$ and $s = [(s_{\eps})_{\eps}]$ is $S$-inverse of $r$.

 In order to conclude the proposition, suppose that $(r_{\eps})_{\eps} \notin \Neg$. Then for some $M \in \N$ and $\eps_k \searrow 0$ we have that $\left\vert r_{\eps_k} \right\vert > \eps_k^M$. Thus $r$ is strictly non-zero with respect to $S = \{ \eps_k : k \in \N \}$. On the other hand, if $(r_{\eps})_{\eps} \in \Neg$, then $\left\vert r_{\eps} \right\vert = O(\eps^m)$ for all $m \in \N$ and $\eps$ sufficiently small. Hence $r$ is not strictly non-zero with respect to any $S \subseteq I$ that satisfies $0 \in \overline{S}$. 
\end{proof}

\begin{theorem}
 Let $S \subseteq I$ such that $0 \in \overline{S}$. Then $u \in \G(X)$ is invertible with respect to $S$ (in $\G(X)$) if and only if $u(x)$ is invertible with respect to $S$ for all $x \in \Xct$ (in $\Kt$).
\end{theorem}
\begin{proof}
 Let $v = [(v_{\eps})_{\eps}]$ be an $S$-inverse of $u$ and $r$ as in Definition~\ref{def:Sinv}. Thus $u(x)v(x) = r$ in $\Kt$, which immediately implies the invertibility of all $u(x)$, $x \in \Xct$, with respect to $S$.

 Assume that $u = [(u_{\eps})_{\eps}]$ is not invertible with respect to $S$. As above, it can be shown that invertibility with respect to $S$ is equivalent to being strictly non-zero on $S$. Thus there exists $K \subset\subset X$ such that for each $m \in \N$ we have $\eps_m \in S$ ($\eps_m \searrow 0$) and $x_{\eps_m} \in K$ that satisfy $\left\vert u_{\eps_m} (x_{\eps_m}) \right\vert \leq \eps_m^m$. Without loss of generality we may assume that $(x_{\eps_m})_m$ converges to some $x$ in a chart $(v,V)$, and that for each $n \in \N$ the sequence $(m^n(v(x_{\eps_m}) - v(x)))_m$ is bounded. A modified version of the Special Curve Lemma~\cite[p.~18]{Kriegl1997} allows us to define a generalized point $x = [(x_{\eps})_{\eps}] \in \Xct$ such that the values $x_{\eps_m}$ remain the same (details for this construction can be found in \cite[Cor.~5.2.6 \& Thm.~6.2.12]{Burtscher2009}). Thus $u(x)$ is not strictly non-zero with respect to $S$, a contradiction. 
\end{proof}

 Given a non-zero multiplicative linear functional $\nu: \G(X) \rightarrow \Kt$, it is easy to see that $\ker(\nu) \lhd \G(X)$---although not maximal---is maximal with respect to the following property:

\begin{definition}
 An ideal $\mathcal{I} \lhd \G(X)$ is called \emph{maximal with respect to $\Kt 1$} if {$\mathcal{I} \cap \Kt 1 = \{ 0 \}$} and any other such ideal $\mathcal{J} \supseteq \mathcal{I}$ equals $\mathcal{I}$.
\end{definition}

 With these results at hand, the compactly bounded generalized points in $X$ can be identified with non-zero multiplicative linear functionals $\nu$. The following key argument can be proved analogously to \cite[Prop.~4.4 \& Thm.~4.5]{Vernaeve2010}.

\begin{theorem} \label{thm:nzmlf}
 Let $\nu: \G(X) \rightarrow \Kt$ be a non-zero multiplicative linear functional. Then there exists a unique $x \in \Xct$ such that \[\nu(u) = u(x) \quad \forall u \in \G(X).\]
\end{theorem}

 It is evident that the final construction of the manifold-valued generalized function in \cite[Thm.~5.1]{Vernaeve2010} also holds for Colombeau generalized functions with smooth parametrization (see \cite[Thm.~6.4.1]{Burtscher2009}): Given an isomorphism $\varPhi: \G(X) \rightarrow \G(Y)$ and a point $y \in \Yct$ one obtains a unique $x \in \Xct$ that satisfies $\operatorname{ev}_y \circ \varPhi = \operatorname{ev}_x$. This identification extends to a map $\varphi \in \G[Y,X]$ with the required properties, and we obtain:

\begin{theorem} \label{thm:gen}
 Let $X$ and $Y$ be manifolds that are Hausdorff and second countable, and $\varPhi: \G(X) \rightarrow \G(Y)$.
\begin{enumerate}
 \item If $\varPhi$ is a homomorphism (with $\varPhi(1)=1$), then there exists a unique $\varphi \in \G[Y,X]$ such that \[\varPhi(u) = u \circ \varphi \quad \forall u \in \G(X).\]
 \item If $\varPhi$ is an isomorphism, then $\varphi$ as in (i) is invertible and the inverse satisfies $\varPhi^{-1}(v) = v \circ \varphi^{-1}$ for all $v \in \G(Y)$. In this case, $\dim X = \dim Y$.
\end{enumerate}
\end{theorem}

\begin{remark}
 Second countability is crucial in the identification of compactly bounded generalized points $x \in \Xct$ with non-zero multiplicative linear functionals $\psi : \G(X) \rightarrow \Kt$ in Theorem~\ref{thm:nzmlf}, since countable exhaustions by compact sets are used in the argument. Most of the preceding results, however, only require paracompact manifolds.
\end{remark}

\begin{remark} \label{rem:smooth}
 Uniqueness of $\varphi$ is already evident by restricting to the subalgebra $\Cinf(X)$, cf.~\cite[Prop.~3.3]{Kunzinger2003}.
\end{remark}

\section{Isomorphisms of algebras of smooth functions}
\label{sec:4}

 Given an isomorphism $\varPsi: \Cinf(X) \rightarrow \Cinf(Y)$ solely between algebras of smooth functions, our aim is to obtain Theorem~\ref{thm:smooth} via Theorem~\ref{thm:gen}. Remark~\ref{rem:smooth} suggests that `lifting' $\varPsi$ to an isomorphism $\varPhi: \G(X) \rightarrow \G(Y)$ is sensible. The definition of $\varPhi$ is straightforward, namely $\varPhi ([(u_{\eps})_{\eps}]) = [(\varPsi(u_{\eps}))_{\eps}]$ for $u = [(u_{\eps})_{\eps}] \in \G(X)$.
\[ \xymatrix{\hspace{1.1cm} \iota(\Cinf(X)) \subseteq \G(X) \ar@{.>}[rrr]^{\varPhi} &&& \G(Y) \supseteq  \iota(\Cinf(Y)) \hspace{1.1cm} \ar @{>-->} [d]^{\cong} \\
 \Cinf(X) \ar@{^(-->}[u]^{\cong} \ar[rrr]^{\varPsi} &&& \Cinf(Y)} \]
It is, however, not obvious whether this process leads to well-defined generalized functions in $Y$. We will show this for $X$ and $Y$ compact (and Hausdorff) manifolds, starting with a general observation on the continuity of such algebra isomorphisms.

\begin{lemma} \label{invclosed}
 Let $X$ and $Y$ be manifolds. $\Cinf(X)$ is holomorphically closed in the $C^*$-algebra $\Cc(X)$, i.e.\ any $f \in \Cinf(X)$ is invertible in $\Cc(X)$ if and only if it is invertible in $\Cinf(X)$.
\end{lemma}

\begin{definition}
 Let $\A$ be a unital Banach algebra over $\C$. For $a \in \A$ the \emph{spectrum of $a$} is \[ \sigma(a) = \{ \lambda \in \C : a - \lambda 1 \text{ is not invertible in } \A \} \] and the \emph{resolvent set of $a$} is $ \rho(a) = \C \setminus \sigma(a).$ The \emph{spectral radius of $a$} is denoted by $r(a) = \max \{ | \lambda | : \lambda \in \sigma(a) \}$.
\end{definition}

\begin{theorem} \label{cont}
 Let $X$ and $Y$ be compact manifolds and let $\varPsi: \Cinf(X) \rightarrow \Cinf(Y)$ be an algebra isomorphism. Then $\varPsi$ is a homeomorphism with respect to the natural topologies on $\Cinf(X)$ and $\Cinf(Y)$ (i.e., uniform convergence in all derivatives).
\end{theorem}

\begin{proof}
 Let $f \in \Cinf(X) \subseteq \Cc(X)$. The spectrum of $f$ in $\Cc(X)$ is $\sigma(f) = f(X)$. Thus $r(f) = \Vert f \Vert_{\infty}$.

 The algebra isomorphism $\varPsi$ preserves the spectrum and resolvent set of $f$, and hence also the $\sup$-norm:
\begin{equation} \label{eq:norm}
 \Vert f \Vert_{\infty} = r(f) = r(\varPsi(f)) = \Vert \varPsi(f) \Vert_{\infty} \quad \forall f \in \Cinf(X).
\end{equation}

 We consider the semi-norms
\[ p_{D_1,...,D_k}(f) = \Vert D_1 ... D_k f \Vert_{\infty}, \] where $D_1,...,D_k \in \Der(\Cinf(X))$ are derivations of the algebra $\Cinf(X)$.

 For $D \in \Der(\Cinf(Y))$ we denote by $\varPsi^*(D)$ the pullback of $D$ under $\varPsi$, i.e. $$\varPsi^*(D)(f) = \varPsi^{-1}(D(\varPsi(f))) \quad \forall f \in \Cinf(X).$$ We need to show that $\varPsi^*(D)$ is again a derivation. Clearly, $\varPsi^*(D)$ is $\R$-linear. Moreover, let $f,g \in \Cinf(X)$. Then
$\varPsi(\varPsi^*(D)(fg)) = D(\varPsi(fg)) = D(\varPsi(f)\varPsi(g)) = \varPsi(f)D(\varPsi(g)) + D(\varPsi(f))\varPsi(g)$, and since $\varPsi$ is bijective we have that $\varPsi^*(D)(fg) = f \varPsi^*(D)(g) + \varPsi^*(D)(f) g$.

 Moreover, $\Vert D(\varPsi(f)) \Vert_{\infty} = \Vert \varPsi(\varPsi^*(D)(f)) \Vert_{\infty} \stackrel{\text{\eqref{eq:norm}}}{=} \Vert \varPsi^*(D)(f) \Vert_{\infty}$. Iterating this procedure we find that
\begin{equation} \label{eq:sn}
 p_{D_1,...,D_k}(\varPsi(f)) = p_{\varPsi^*(D_1),...,\varPsi^*(D_k)}(f) \quad \forall f \in \Cinf(X).
\end{equation}
 Thus $\varPsi$ is continuous, and we can argue analogously for $\varPsi^{-1}$. 
\end{proof}

 Using this result we can derive the characterization of algebra isomorphisms $\varPsi: \Cinf(X) \rightarrow \Cinf(Y)$ from that of algebra isomorphisms $\varPhi: \G(X) \rightarrow \G(Y)$.

\begin{lemma} \label{coliso}
 Let $X$ and $Y$ be compact manifolds and $\varPsi: \Cinf(X) \rightarrow \Cinf(Y)$ an algebra isomorphism. Then $\varPhi: \G(X) \rightarrow \G(Y)$ defined by $\varPhi([(u_{\eps})_{\eps}]) = [(\varPsi(u_{\eps}))_{\eps}]$ is a well-defined algebra isomorphism.
\end{lemma}

\begin{proof}
 To begin with note that $\varPhi(\Cinf(I \times X,\K)) \subseteq \Cinf(I \times Y,\K)$. By the definition of $\EM$ and $\Neg$ and the continuity of $\varPsi$ by Theorem~\ref{cont}, it follows that $\varPhi(\EM(X)) \subseteq \EM(Y)$ and $\varPhi(\Neg(X)) \subseteq \Neg(Y)$: For $(u_{\eps})_{\eps} \in \EM(X)$, any $K \subset\subset Y$ and arbitrary $D_1,...,D_k \in \Der(\Cinf(Y))$ there exists $L \in \N$ such that as $\eps \rightarrow 0$
\begin{eqnarray*}
 \sup_{x \in K} \vert D_1 ... D_k \varPsi(u_{\eps})(x) \vert &\leq& p_{D_1,...,D_k}(\varPsi(u_{\eps})) \\
&\stackrel{\text{\eqref{eq:sn}}}{=}& p_{\varPsi^*(D_1),...,\varPsi^*(D_k)}(u_{\eps}) = O(\eps^{-L}).
\end{eqnarray*}
 Similarly, for $(v_{\eps})_{\eps} \in \Neg(X)$, any $K \subset\subset Y$ and any $m \in \N$, we have that $\sup_{x \in K} \vert \varPsi(v_{\eps})(x) \vert = O(\eps^m)$ as $\eps \rightarrow 0$. Therefore $\varPhi$ is well-defined.

 Moreover, $\varPhi^{-1}([(v_{\eps})_{\eps}]) = [(\varPsi^{-1}(v_{\eps}))_{\eps}]$, so $\varPhi$ is an algebra isomorphism. 
\end{proof}

\begin{theorem}
 Let $X$ and $Y$ be compact manifolds, and $\varPsi: \Cinf(X) \rightarrow \Cinf(Y)$ an algebra isomorphism. Then there exists a unique diffeomorphism $\psi: Y \rightarrow X$ such that $$\varPsi(f) = f \circ \psi \quad \forall f \in \Cinf(X).$$
\end{theorem}

\begin{proof}
 By Lemma~\ref{coliso} there exists an algebra isomorphism $\varPhi: \G(X) \rightarrow \G(Y)$ such that $\left. \varPhi \right\vert_{\Cinf(X)} = \varPsi$ (we omit the natural embeddings $\iota_{\_}: \Cinf(\_) \hookrightarrow \G(\_)$ throughout). By Theorem~\ref{thm:gen} there exists $\varphi \in \G[Y,X]$ such that
\begin{equation} \label{eq:psi}
 \varPhi(u) = u \circ \varphi \quad \forall u \in \G(X) \quad \text{and} \quad \varPhi^{-1}(v) = v \circ \varphi^{-1} \quad \forall v \in \G(Y).
\end{equation}
 It remains to be shown that $\varphi$, in fact, is a diffeomorphism. By the above, $f \circ \varphi \in \Cinf(Y)$ for all $f \in \Cinf(X)$. By the Whitney Embedding Theorem~\cite{Hirsch1976} there exists a smooth embedding $j: X \rightarrow \R^m$ for some $m \in \N$. 

 By \cite[Def.~2.1 \& Prop.~2.2]{Kunzinger2009} we know that $j \circ \varphi \in \widetilde{\mathcal{G}}[Y,j(X)] \subseteq \G(Y)^m$. Let $\pr_i : \R^m \rightarrow \R$ be the $i$-th projection. Since $\pr_i \circ j \in \Cinf(X)$ we have that $(\pr_i \circ j) \circ \varphi = \varPsi(\pr_i \circ j) \in \Cinf(Y)$ for all $1 \leq i \leq m$, and therefore $j \circ \varphi \in \Cinf(Y,\R^m)$. Apart from $(j \circ \varphi)_{\eps}$ itself, $\widetilde{\varphi} = j \circ \varphi \in \widetilde{\mathcal{G}}[Y,j(X)]$ also possesses a representative $(\widetilde{\varphi}_{\eps})_{\eps}$ that satisfies $\widetilde{\varphi}_{\eps}(Y) \subseteq j(X)$ for all $\eps$. Thus for all $p \in Y$ we have that $\widetilde{\varphi}_{\eps}(p) \rightarrow (j \circ \varphi)(p)$ as $\eps \rightarrow 0$. Since $j(X)$ is closed ($X$ is compact and the Whitney embedding is smooth) this implies that $(j \circ \varphi)(p) \in j(X)$ for all $p \in Y$. Summing up, $\varphi = j^{-1} \circ (j \circ \varphi) \in \Cinf(Y,X)$. By symmetry also $\varphi^{-1} \in \Cinf(X,Y)$. Thus $\varphi$ is the required $\psi$. 
\end{proof}

\begin{remark}
 Note that the above arguments heavily rely on the spectral radius formula, $r(f) = \left\Vert f \right\Vert_{\infty}$, on compact manifolds, which allows us to relate algebraic structures on one side to analytic/geometric structures on the other side as required in Theorem~\ref{cont}. A generalization to non-compact manifolds seems feasible but more involved.
\end{remark}

\section*{Acknowledgments}
 The author would like to thank Michael Kunzinger for several helpful discussions.

\end{document}